\documentclass[11pt]{amsart}

\usepackage{ amsmath,amssymb,verbatim,graphicx,amsthm}
\usepackage[top=1in, bottom=1in, left=1in, right=1in]{geometry}
\usepackage{cite}
 \usepackage{url}

 \allowdisplaybreaks[1]

\usepackage[sc]{mathpazo}
\usepackage[T1]{fontenc}

\newcommand{\vertiii}[1]{{\left\vert\kern-0.25ex\left\vert\kern-0.25ex\left\vert #1
    \right\vert\kern-0.25ex\right\vert\kern-0.25ex\right\vert}}

\newcommand{\RR}{\mathbb{R}}

\newcommand{\ZZ}{\mathbb{Z}}

\newcommand{\NN}{\mathbb{N}}
\newcommand{\TT}{\mathbb{T}}

\newcommand{\EE}{\mathbb{E}}

\theoremstyle{plain}

\newtheorem{theorem}{Theorem}[section]

\newtheorem{lemma}[theorem]{Lemma}

\newtheorem{mydef}[theorem]{Definition}

\usepackage{setspace}
\begin{document}

\title{Pointwise double recurrence and nilsequences}

\author{Idris Assani}
\address{Department of Mathematics, The University of North Carolina at Chapel Hill,
Chapel Hill, NC 27599}
\email{assani@math.unc.edu}
\urladdr{http://www.unc.edu/math/Faculty/assani/} 


\begin{abstract}
Consider  a system $(X, \mathcal{F}, \mu, T)$, bounded functions $f_1, f_2 \in L^\infty(\mu)$ and $a,b \in \ZZ.$ We show that  there exists a set of full measure $X_{f_1, f_2}$ in $X$ such that for all $x \in X_{f_1, f_2}$ and for every nilsequence $b_n$  , the averages
\[ \frac{1}{N} \sum_{n=1}^N f_1(T^{an}x)f_2(T^{bn}x)b_n \]
converge. We will show that this can be deduced from the classical Wiener-Wintner theorem for the double recurrence theorem. Together with the past work on this subject, we will show that several statements regarding the extension of the double recurrence theorem are equivalent.
\end{abstract}

\maketitle

\section{Introduction}
 Throughout this paper we denote a system by a quadruple $(X, \mathcal{F}, \mu, T)$, where $(X, \mathcal{F}, \mu)$ is a probability measure space, and $T$ is a measure preserving transformation on it.  Without loss of generality we can assume that the system is standard meaning that $X$ is a compact metric space, $T$ a homeomorphism and $\mathcal{F}$ the $\sigma$-field of the Borelian subsets of $X$ . The purpose of this note is to answer  a question raised by B. Weiss relative to the extensions of J. Bourgain double recurrence theorem obtained recently in \cite{WWDR}, \cite{WWDR_poly} and in \cite{NewUnivWeight_Norm}.
 We start by recalling the definition of a nilsequence as given in \cite{HostKraUniformity}.
  \begin{mydef}
   Let $ a _n$ be a sequence of complex numbers . This sequence is a  $k-\text{step}$ basic nilsequence if it can be written as $F(g^n\Gamma)$ , where $F\in \mathcal{C}(X)$ , $X = G/\Gamma$, G is a $k$-step nilpotent group, $\Gamma$ a discrete cocompact subgroup, and $g\in G.$
   \end{mydef}
   \begin{mydef}
  A sequence $a_n$ is a $k$-step nilsequence if it is a uniform limit of $k$-step  basic nilsequence.
 \end {mydef}
 We shall also recall the classical Wiener-Wintner ergodic theorem from 1941.
 \begin{theorem}[Wiener-Wintner Ergodic Theorem, \cite{WW}]\label{WW}
 Let $(X, \mathcal{F}, \mu, T)$ be a measure-preserving system, and $f \in L^1(\mu)$. Then there exists a set of full measure $X_f$ such that for any $x \in X_f$ and for any $t \in [0, 1)$, the averages
 \[ \frac{1}{N} \sum_{n=1}^N f(T^nx)e^{2\pi int} \]
 converge. 
 \end{theorem}
 The original proof of this theorem by N. Wiener and A. Wintner contained errors, but the argument was later fixed by multiple scholars, including H. Furstenberg \cite{Furstenberg_processes}. The Wiener-Wintner theorem played a key role in J. Bourgain's double recurrence theorem \cite{BourgainDR}. In fact, Bourgain announced that when the transformation is ergodic, then the uniform version of the Wiener-Wintner theorem holds.
 \begin{theorem}[Uniform Wiener-Wintner ergodic theorem]\label{uniformWW}
 Let $(X, \mathcal{F}, \mu, T)$ be an \textbf{ergodic} system, and $f \in L^1(\mu)$. Then the following statements are equivalent.
 \begin{enumerate}
 \item The function $f$ belongs to the orthogonal complement of the closed linear span of the eigenfunctions of $T$.
 \item We have
 \[ \lim_{N \to \infty} \sup_{t \in [0, 1)} \left|\frac{1}{N} \sum_{n=1}^N f(T^nx)e^{2\pi int} \right| = 0. \]
 \end{enumerate}
 \end{theorem}
 We note that the uniformity does not need to hold if $T$ is not ergodic. The proof of the Wiener-Wintner as well as the uniform Wiener-Wintner theorems can be found in \cite[Chapter 2]{AssaniWWET}. In 2014, the author, D. Duncan, and R. Moore generalized Theorems $\ref{WW}$ and $\ref{uniformWW}$ for the double recurrence averages by using Bourgain's result \cite{BourgainDR}.
 \begin{theorem}[Double recurrence Wiener-Wintner ergodic theorem, \cite{WWDR}]\label{WWDRThm}
 	Let $(X, \mathcal{F}, \mu, T)$ be a measure-preserving system, and $f_1, f_2 \in L^\infty(\mu)$. Then there exists a set of full measure $X_{f_1, f_2}$ such that for any $x \in X_{f_1, f_2}$, for any nonzero distinct integers $a$ and $b$, and for any $t \in [0, 1)$, the averages
 	\[ W_N(f_1, f_2, x, t) = \frac{1}{N} \sum_{n=1}^N f_1(T^{an}x) f_2(T^{bn}x) e^{2\pi int}\]
 	converge. Furthermore, if $T$ is ergodic, and if either $f_1$ or $f_2$ belongs to the orthogonal complement of the second Host-Kra-Ziegler factor $\mathcal{Z}_2$, we have
 	\[\lim_{N \to \infty} \sup_{t \in [0, 1)} |W_N(f_1, f_2, x, t)| = 0.\]
 \end{theorem}
 We refer the readers to \cite{HostKraNEA} and \cite{Ziegler} for the constructions of the Host-Kra-Ziegler factors. We note that Theorem $\ref{WWDRThm}$ has been extended by the author and R. Moore to the polynomial Wiener-Wintner theorem. Polynomial Wiener-Wintner averages for a single function has been studied by E. Lesigne and N. Frantzikinakis. In 1990, Lesigne showed that given an ergodic system $(X, \mathcal{F}, \mu, T)$, and a function $f \in L^1(\mu)$, then there exists a set of full measure $X_f$ such that for any $x \in X_f$ and any polynomial $p$ with real coefficients, the averages
 \[ \frac{1}{N} \sum_{n=1}^N f(T^nx)e^{2\pi i p(n)}\]
 converge \cite{Lesigne90}. He later showed that if $T$ is totally ergodic (i.e. every power of $T$ is ergodic), and $f$ belongs to the orthogonal complement of the $k$-th Abramov factor, then these averages converge to $0$ \cite{Lesigne93}, which was later extended by Frantzikinakis that the averages converge to $0$ uniformly \cite{Fran06}. Frantzikinakis also showed that the assumption $T$ being totally ergodic cannot be dropped (i.e. $T$ cannot be just ergodic), or the uniformity does not hold. Using the Host-Kra-Ziegler factor, however, we can show that the uniformity indeed holds.
 \begin{theorem}[Polynomial double recurrence Wiener-Wintner theorem, \cite{WWDR_poly}]\label{WWDR_poly_thm}
 Let $(X, \mathcal{F}, \mu, T)$ be a measure-preserving system, and $f_1, f_2 \in L^\infty(\mu)$. Then there exists a set of full measure $X_{f_1, f_2}$ such that for any $x \in X_{f_1, f_2}$, for any nonzero distinct integers $a$ and $b$, and for any polynomial $p$ with real coefficients, the averages
  	\[ W_N(f_1, f_2, x, p) = \frac{1}{N} \sum_{n=1}^N f_1(T^{an}x) f_2(T^{bn}x) e^{2\pi ip(n)}\]
  	converge. Furthermore, if $T$ is ergodic, and if either $f_1$ or $f_2$ belongs to the orthogonal complement of the $k+1$-th Host-Kra-Ziegler factor $\mathcal{Z}_{k+1}$, and if $\RR_k[\xi]$ denotes the collection of all the degree-$k$ polynomials with real coefficients, we have
  	\[\lim_{N \to \infty} \sup_{p \in \RR_k[\xi]} |W_N(f_1, f_2, x, p)| = 0.\]
 \end{theorem}
 
   The interest in nilsequences appears in several papers linked to problems in number theory see for instance the papers by B. Green and T. Tao , \cite{GreenTaoPrimesAP}, \cite{GreenTaoQuadratic}, \cite{GreenTaoPrimesLinear} and the paper by V. Bergelson, B. Host and B. Kra \cite{BHK_nilseq}.  Our interest in these sequences comes from the simple observation that sequences of the form $e^{2\pi in t}$ or $e^{2\pi i P(n)}$ where $P$ is a real polynomial with integer coefficients, are nilsequences and that these weights were used to obtain Wiener Wintner extension of J. Bourgain result in \cite{WWDR} and \cite{WWDR_poly}.
   One of our main results, which we apply Bourgain's double recurrence theorem \cite{BourgainDR} to prove, is the following:
   \begin{theorem}\label{mainResult}
   Let $(X, \mathcal{F}, \mu, T)$ be a system and $f_1$, $f_2$ bounded measurable functions. There exists a set of full measure $X_{f_1, f_2}$ such that for all $x\in X_{f_1, f_2}$ and for every nilsequence $b_n$ the averages \[ \frac{1}{N} \sum_{n=1}^N f_1(T^{an}x)f_2(T^{bn}x)b_n \]  converge.
   \end{theorem}
  Thus the answer to B. Weiss question is positive. Furthermore, combining all the results from \cite{WWDR}, \cite{WWDR_poly}, \cite{NewUnivWeight_Norm}, and Theorem $\ref{mainResult}$, the following statements are known to be equivalent. In particular, all of them can be deduced from the double recurrence Wiener-Wintner theorem (Theorem $\ref{WWDRThm}$). 
 \begin{theorem}\label{equivalence}
 Let $(X, \mathcal{F}, \mu, T)$ be a measure-preserving system, and $f_1, f_2 \in L^\infty(\mu)$. Then the following statements are equivalent.
 \begin{enumerate}
 \item There exists a set of full measure $X_{f_1, f_2}^1$ such that for any $x \in X_{f_1, f_2}^1$, for any nonzero distinct integers $a$ and $b$, and for any  $t \in [0, 1)$, the averages
   	\[ \frac{1}{N} \sum_{n=1}^N f_1(T^{an}x) f_2(T^{bn}x) e^{2\pi int}\]
   	converge.
 \item There exists a set of full measure $X_{f_1, f_2}^2$ such that for any $x \in X_{f_1, f_2}^2$, for any nonzero distinct integers $a$ and $b$, and for any polynomial $p$ with real coefficients, the averages
   	\[ \frac{1}{N} \sum_{n=1}^N f_1(T^{an}x) f_2(T^{bn}x) e^{2\pi i p(n)}\]
   	converge.
  \item There exists a set of full measure $X_{f_1, f_2}^3$ such that for all $x\in X_{f_1, f_2}^3$, for any nonzero distinct integers $a$ and $b$, and for every nilsequence $b_n$ the averages \[ \frac{1}{N} \sum_{n=1}^N f_1(T^{an}x)f_2(T^{bn}x)b_n \] converge.
  \item There exists a set of full measure $X_{f_1, f_2}^4$ such that for any $x \in X_{f_1, f_2}^4$, for any nonzero distinct integers $a$ and $b$, and for any other measure-preserving system $(Y, \mathcal{G}, \nu, S)$ and functions $g_1, g_2, \ldots, g_k$, the averages
  \[ \frac{1}{N} \sum_{n=1}^N f_1(T^{an}x)f_2(T^{bn}x) \prod_{i=1}^k g_i \circ S^{in} \]
  converge in $L^2(\nu)$.
 \end{enumerate}
 \end{theorem}
 Certainly, we can take $X_{f_1, f_2} = \bigcap_{i=1}^4 X_{f_1, f_2}^i$, and this will be the universal set of full measure that satisfies all of (1) through (4).
 \begin{proof}[Proof of Theorem $\ref{equivalence}$] We consider the following cases.
 \begin{itemize}
 \item \underline{(1) $\Leftrightarrow$ (2)}: The forward implication is proved in \cite{WWDR_poly}. The reverse implication is immediate by setting $p(n) = tn$ for any $t \in [0, 1)$.
 \item \underline{(1) $\Leftrightarrow$ (3)}: The forward implication is Theorem $\ref{mainResult}$, which is proved in this paper. The reverse implication is immediate since $e^{2\pi int}$ is a one-step nilsequence.
 \item \underline{(1) $\Leftrightarrow$ (4)}: The forward implication is proved in \cite{NewUnivWeight_Norm}. The reverse implication can be shown by setting $Y = \TT$, $\mathcal{G}$ to be the Borel $\sigma$-algebra of $\TT$, where $\nu$ is a Borel probability measure on $\TT$, and $S = R_t$ is a rotation by $t$ (i.e. $R_t(e^{2\pi i \alpha}) = e^{2\pi i (\alpha+t)}$), and $g(y) = e^{2\pi i y}$. 
 \end{itemize}
 \end{proof}
\section{Proof of the main theorem}
   The main ingredients in the proof are:
   \begin{enumerate}
   \item Some properties of nilsequences given in \cite{HostKraUniformity}.
   \item Elements in the proof of the pointwise convergence of the averages along the cubes established in
   \cite{AssaniCubes}.
   \end{enumerate}
   \subsection{Preliminaries}
  

We extract from \cite{HostKraUniformity} some properties of nilsequences that we will be using. First we need some notations. Let $a_n$ be a bounded sequence of real numbers . For every $k\geq 1$, points of $\ZZ^k$ are written $h= (h_1,h_2, ..., h_k)$. For $\epsilon = (\epsilon_1, \epsilon_2, ..., \epsilon_k)\in \{0,1\}^k $ and $h= (h_1, h_2, ..., h_k) \in\ZZ^k$ , we define $|\epsilon| = \sum_{l=1}^k \epsilon_l$ and  we denote by  $\epsilon.h$ the dot product
$\sum_{l=1}^k \epsilon_l.h_l$ (so $|\epsilon| = \epsilon . \epsilon$).
The next lemmas list the properties we seek.
\begin{lemma}\label{L1}
 Let $a_n$ and $b_n$ two nilsequences of order respectively $k_1$ and $k_2$ then
 \begin{enumerate}
 \item the sequence $c_n = a_nb_n$ is a nilsequence of order $\max\{k_1, k_2\}$
 \item for each $k$ nilsequence $\alpha_n$ the averages $\frac{1}{N}\sum_{n=0}^{N-1} \alpha_n$ converge.
 \end{enumerate}
\end{lemma}
\begin{proof}
 The first part follows immediately from the nilpotent structure of the product of two homogeneous spaces generating the sequences $a_n$ and $b_n.$ The second part is a consequence of the unique ergodicity of the  system associated with the $k$-step nilsequence $\alpha_n.$
\end{proof}
\begin{lemma}\label{L2}
  Let us fix $k\in \NN$. Assume that the real bounded sequence $a_n$ is such that
  $$c_h= \lim_N\frac{1}{N}\sum_{n=0}^{N-1} \prod_{\epsilon\in \{0,1\}^k} a_{n+ h.\epsilon}$$ exists. Then
  \begin{enumerate}
 \item $\lim_{H}\frac{1}{H^k}\sum_{h_1, ...,h_k= 0}^{H-1} c_h$  exists and is nonnegative. Therefore
 $ \| a\|_k = \left( \lim_{H}\frac{1}{H^k}\sum_{h_1, ...,h_k= 0}^{H-1} c_h\right)^{1/2^k}$ is well defined.
 \item if $\| a\|_k=0$ then for any $k-1$ step nilsequence $b_n$ we have
 $\lim_N\frac{1}{N} \sum_{n=0}^{N-1} a_n b_n = 0.$
  \end{enumerate}
\end{lemma}
\begin{proof}
 It is a consequence of the Proposition 2.2 and Corollary 2.14 in \cite{HostKraUniformity}.  In that paper the authors defined what they call local seminorms with respect to a sequence of intervals $I_j$ in $\ZZ$ with length tending to infinity. In this paper we only focus on the sequence of intervals in $\NN$ of the form $[0,  N-1].$
So the seminorm $\||a\||_k$ in our lemma corresponds to the local semi norm $\||a|\|_{I,k}$ with $I= (I_{j})$ where $I_j = [0, j-1].$
 Proposition 2.2 in \cite{HostKraUniformity} says that part (1) is true, while Corollary 2.14 from the same reference tells us that for any $\delta  > 0$, there exists a constant $c$ that depends on $(b_n)$ and $\delta$ such that $\limsup_N \left|\frac{1}{N} \sum_{n=0}^{N-1} a_nb_n\right| \leq c.0 + \delta \|a\|_{\infty} = \delta \|a\|_{\infty}.$  From this, part 2 of the lemma follows.
 \end{proof}
 \noindent{\bf Remarks:}
    In order to eliminate possible confusion between the local semi norms $\||a|\|_k$ and the similar notation for the Gowers Host Kra semi-norms we will denote by $\||a|\|_k$  the local semi norm and by $\mathbf{\vertiii{f}_k}$ the GHK semi-norms of a function $f$.
 \subsection{Joinings}
  First we need a lemma allowing to define the limit in Bourgain double recurrence theorem as the integral of the functions with respect to a joining.
   \begin{lemma}\label{L3}
     Given a standard dynamical system $(X, \mathcal{F}, \mu, T)$ and $a, b\in \ZZ$ , $a\neq b$, for $\mu$ a.e. $x$ there exists a joining
     $\mu_x$ defined on $((X\times X), \mathcal{F}^2)$ which is $T^a\times T^b$ invariant such that for any continuous function $f\otimes g$, we have
     \begin{enumerate}
     \item  $\lim_N \frac{1}{N}\sum_{n=0}^{N-1} f(T^{an}x)g(T^{bn}x) = \int f\otimes g d\mu_x$
     \item There exists a joining $\omega$ on $((X\times X), \mathcal{F}^2)$ such that
      \begin{equation}\label{join}
      \int \lim_N \frac{1}{N}\sum_{n=0}^{N-1} f(T^{an}x)g(T^{bn}x)d\mu = \int f\otimes g d\mu_xd\mu(x)= \int f\otimes g d\omega .
      \end{equation}
      \item Furthermore if $\mathcal{I}$ denotes the $\sigma$ field of the $T^{b-a}$ invariant subsets of $X$     there exists a kernel $K(x, y)$ such that we have 
      \begin{equation}
       \int f\otimes g d\omega = \int \EE[f |\mathcal{I}]\EE [g |\mathcal{I}] d\mu=   \int\int K(x,y) f(x) g(y)d\mu\otimes\mu
      \end{equation}
     \end{enumerate}

   \end{lemma}
   \begin{proof}
   By using a countably dense set of continuous functions $F_j$ on $X^2$ we can find a set $\tilde{X}$ of full measure in $X$ on which the averages $\frac{1}{N}\sum_{n=0}^{N-1} F_j(T^{an}x, T^{bn}x)$ converge for each $j.$
    By approximation we can conclude that on the same set $\tilde{X}$ we have the convergence of these averages for each continuous function $F$ on $X^2.$ By Riesz representation theorem we can find a measure $\mu_x$
  on $(X\times X, \mathcal{F}^2)$  such that $$\lim_N\frac{1}{N}\sum_{n=0}^{N-1} F( T^{an}x, T^{bn}x) = \int F d\mu_x.$$
  For the particular case where $F = f\otimes g$ we derive the equality
  $$\int \lim_N \frac{1}{N}\sum_{n=0}^{N-1} f(T^{an}x)g(T^{bn}x)d\mu = \int f\otimes g d\mu_xd\mu(x)$$
  by integration with respect to the measure $\mu.$
   It remains to identify the measure $\omega.$ Simple computations, using the measure preserving property of the map $T^a,$ show that
   $$\int \lim_N \frac{1}{N}\sum_{n=0}^{N-1} f(T^{an}x)g(T^{bn}x)d\mu = \int \EE[f |\mathcal{I}]\EE [g |\mathcal{I}] d\mu$$ where $\mathcal{I}$ is the $\sigma-$algebra of invariant subsets for the map $T^{b-a}.$
   In other words $\omega$ is the relatively independent joining over the $\sigma$ algebra $\mathcal{I}.$
    The last part of the theorem follows from Lemma 5.2 in \cite{WWDR}.
   \end{proof}
   \subsection{ Proving that $\||(g_1(T^{an}x)g_2(T^{bn}x))_n|\|_k = 0$}
   Starting with two functions  in $L^{\infty}(\mu)$ that we can assume to be bounded by one  we can decompose them into the sum of their projections onto the Host-Kra-Ziegler factor \cite{HostKraNEA}, \cite{Ziegler } $\mathcal{Z}_{k-1}$ and onto $\mathcal{Z}_{k -1}^{\perp}.$ We denote by $g_1$ and $g_2$ the projections of these functions onto
   $\mathcal{Z}_{k-1}^{\perp}.$ Our goal in this section is to prove the following lemma.
   \begin{lemma}\label{L4}
     With the notations of the previous lemma we have  for $\mu$ a.e. $x\in X,$
     $$\||(g_1(T^{an}x)g_2(T^{bn}x))_n|\|_k = 0.$$
   \end{lemma}
   \begin{proof}
   First we need to check that $c_h(x) $ exists for $a_n(x) =g_1(T^{an}x).g_2(T^{bn}x).$  This follows from J. Bourgain a.e. double recurrence result \cite{BourgainDR}.  Indeed  the quantity
   $ \prod_{\epsilon\in \{0,1\}^k} a_{n+ h.\epsilon}(x)$ can be written as the product of two functions
   $G^1_{h_1,h_2,..., h_k}(T^{an}x)G^2_{h_1,h_2, ...,h_k}(T^{bn}x).$ Therefore  for $\mu$ a.e. $x$  the limit of the averages
     $$\lim_N\frac{1}{N}\sum_{n=0}^{N-1} G^1_{h_1,h_2,..., h_k}(T^{an}x)G^2_{h_1,h_2, ...,h_k}(T^{bn}x)$$
     exists.
      To be more explicit and for the simplicity of the notation we can look at the case $k=3.$ The same ideas will give the proof for the case $k>3.$
      The product  $ \prod_{\epsilon\in \{0,1\}^k} a_{n+ h.\epsilon}(x)$ is equal to the product of
      \begin{align*}
      &G^1_{h_1,h_2, h_3}(T^{an}x) \\
      &= g_1(T^{an}x)g_1(T^{a(n+h_1)}x)g_1(T^{a(n+h_2)}x)g_1(T^{a(n+h_3)}x)g_1(T^{a(n+(h_1+h_2)}x)\\
      &g_1(T^{a(n+(h_1 + h_3)}x)g_1(T^{a(n+ h_2 + h_3)}x) g_1(T^{a(n+h_1 + h_2+ h_3)}x)
      \end{align*}
        and
       \begin{align*}
      &G^2_{h_1,h_2, h_3}(T^{bn}x) \\
      &= g_2(T^{bn}x)g_2(T^{b(n+h_1)}x)g_2(T^{b(n+h_2)}x)g_2(T^{b(n+h_3)}x)g_2(T^{b(n+(h_1+h_2)}x)\\
      &g_2(T^{b(n+(h_1 + h_3)}x)g_2(T^{b(n+ h_2 + h_3)}x) g_2(T^{b(n+h_1 + h_2+ h_3)}x).
       \end{align*}
       By Lemma \ref{L2} we have
       \begin{enumerate}
     \item $\lim_H\frac{1}{H^3} \sum_{h_1, h_2, h_3 = 0}^{H-1} c_h(x) \geq 0$
     \item $\|| g_1(T^{an}x)g_2(T^{bn}x)|\|_3 = \left(\lim_H\frac{1}{H^3} \sum_{h_1, h_2, h_3 = 0}^{H-1} c_h(x)\right)^{1/2^k}$
        \end{enumerate}
       Our  goal is to show that
       \begin{equation}\label{equat}
        \limsup_H \left(\frac{1}{H^3} \sum_{h_1, h_2, h_3 = 0}^{H-1} c_h(x)\right) = 0.
       \end{equation}
        This last equation would certainly suffice to prove Lemma \ref{L4}.
          To establish (\ref{equat}) we will show that
           \begin{equation}
            \int  \limsup_H \left(\frac{1}{H^3} \sum_{h_1, h_2, h_3 = 0}^{H-1} c_h(x)\right) d\mu = 0.
           \end{equation}
        To this end we use Lemma \ref{L3}
         \begin{align*}
          &\int  \limsup_H \left(\frac{1}{H^3} \sum_{h_1, h_2, h_3 = 0}^{H-1} c_h(x)\right) d\mu =
          \int  \lim_H \left(\frac{1}{H^3} \sum_{h_1, h_2, h_3 = 0}^{H-1} c_h(x)\right) d\mu \\
         & = \int \lim_H \left(\frac{1}{H^3} \sum_{h_1, h_2, h_3 = 0}^{H-1}\int G^1_{h_1,h_2, h_3}\otimes   G^2_{h_1,h_2, h_3}d\mu_x\right) d\mu(x)\\
          &= \lim_H\int  \int \frac{1}{H^3} \sum_{h_1, h_2, h_3 = 0}^{H-1} G^1_{h_1,h_2, h_3}\otimes   G^2_{h_1,h_2, h_3}d\mu_x\ d\mu(x)  \, \text{ by the Dominated convergence theorem}\\
          & = \lim_H  \int \frac{1}{H^3} \sum_{h_1, h_2, h_3 = 0}^{H-1}\ G^1_{h_1,h_2, h_3}(x) G^2_{h_1,h_2, h_3}(x) d\omega \, \text{ by Lemma \ref{L3}}\\
          &= \lim_H \int \int K(x,y)  \frac{1}{H^3} \sum_{h_1, h_2, h_3 = 0}^{H-1}G^1_{h_1,h_2, h_3}(x) G^2_{h_1,h_2, h_3}(y) d\mu\otimes d\mu(x,y)
          \end{align*}
          The quantities $\frac{1}{H^3} \sum_{h_1, h_2, h_3 = 0}^{H-1}\ G^1_{h_1,h_2, h_3}(x) G^2_{h_1,h_2, h_3}(y)$ represent averages along cubes of order 3.  The pointwise estimates obtained for these averages in Lemma 6  in \cite{AssaniCubes} gives us the inequality
          \begin{align*}
        & \left(\frac{1}{H^3} \sum_{h_1, h_2, h_3 = 0}^{H-1}\ G^1_{h_1,h_2, h_3}(x) G^2_{h_1,h_2, h_3}(y)\right)^2 \\
        & \leq C  \frac{1}{H}\sum_{h_2=0}^{H-1} \sup_t\left|\frac{1}{H}\sum_{h_1= 0}^{2(H-1)}
         g_1(T^{ah_1}x)g_2(T^{bh_1}y)g_1(T^{a(h_1+h_2)}x)g_2(T^{b(h_1+h_2)}y)e^{2\pi ih_1 t}\right|^2 \\
         &\leq 2C \frac{1}{H}\sum_{h_2=0}^{H-1} \sup_t\left|\frac{1}{2(H-1)}\sum_{h_1=0}^{2(H-1)}
         g_1(T^{ah_1}x)g_2(T^{bh_1}y)g_1(T^{a(h_1+h_2)}x)g_2(T^{b(h_1+h_2)}y)e^{2\pi ih_1 t}\right|^2
          \end{align*}
        where $C$ is an absolute constant.
        We would like to prove that under the assumption made on the the functions $g_1$ and $g_2$ the last term  converge to zero. To this end we use some of the estimates made in \cite{AssaniCubes}. As $x$ and $y$ will be fixed throughout these estimates, to simplify the notations we simply write
        $\alpha_{h_1} = g_1(T^{ah_1}x)g_2(T^{bh_1}y)$ and  so $\alpha_{h_1+h_2} = g_1(T^{a(h_1 +h_2)}x) g_2(T^{b(h_1 +h_2}y) .$
         We use the van der Corput lemma (see \cite{KuipersNiederreiter}) to derive that  for $(K+1)^2 <H$  we have
         \begin{align*}
          & \sup_t\left|\frac{1}{H}\sum_{h_1}^{H-1} \alpha_{h_1}\alpha_{h_1 +h_2} e^{2\pi ih_1t}\right|^2 \\
          &\leq \frac{C}{K} + \frac{C}{K} \sum_{k=1}^K\left| \frac{1}{H}\sum_{h_1=0}^{H-k-1}\alpha_{h_1}\alpha_{h_1+k}\alpha_{h_1 + h_2}\alpha_{h_1 + h_2 +k}\right|.
         \end{align*}
          As a consequence because the sequence $\alpha_n$ is assumed bounded by one we have
          \begin{align*}
          &\frac{1}{H}\sum_{h_2 =0}^{H-1}\sup_t\left|\frac{1}{H}\sum_{h_{1}=0}^{H-1} \alpha_{h_1}\alpha_{h_1 +h_2} e^{2\pi ih_1t}\right|^2 \\
          &\leq \frac{C}{K} + \frac{C}{K} \sum_{k=1}^K \frac{1}{H}\sum_{h_2=0}^{H-1}\left|\frac{1}{H} \sum_{h_1=0}^{H-1}\alpha_{h_1}\alpha_{h_1+k}\alpha_{h_1 + h_2}\alpha_{h_1 + h_2 +k}\right| \\
          & \leq  \frac{C}{K} + \left(\frac{C}{K} \sum_{k=1}^K \left(\frac{1}{H}\sum_{h_2=0}^{H-1}\left|\frac{1}{H} \sum_{h_1=0}^{H-1}\alpha_{h_1}\alpha_{h_1+ k}\alpha_{h_1 + h_2}\alpha_{h_1 + h_2 +k}\right|^2\right)^{1/2}\right) \\ & \text{ (by Cauchy Schwarz inequality) } 
          \end{align*}
    Now we can apply part 2 of the remarks 3 in \cite{AssaniCubes} . It gives us the following estimate
      \begin{align*}
      \left(
          \frac{1}{H}\sum_{h_{2} = 0}^{H-1}\left| \frac{1}{H} \sum_{h_{1} =0}^{H-1}  \alpha_{h_1}\alpha_{h_1+k}\alpha_{h_1 + h_2}\alpha_{h_1 + h_2 +k}\right|^2 \right)^{1/2} \leq C \sup_t\left|\frac{1}{H}\sum_{ h_1=0}^{2(H-1)}\alpha_{h_1}\alpha_{h_1+ k}e^{2\pi ih_1t}\right|
          \end{align*}
        Going back to the functions $g_1$ and $g_2$ we have obtained the estimate

        \begin{align*}
        & \left(\frac{1}{H^3} \sum_{h_1, h_2, h_3 = 0}^{H-1}\ G^1_{h_1,h_2, h_3}(x) G^2_{h_1,h_2, h_3}(y)\right)^2 \\
        &\leq 2C \frac{1}{H}\sum_{h_2=0}^{H-1} \sup_t\left|\frac{1}{2(H-1)}\sum_{h_1=0}^{2(H-1)}
         g_1(T^{ah_1}x)g_2(T^{bh_1}y)g_1(T^{a(h_1+h_2)}x)g_2(T^{b(h_1+h_2)}y)e^{2\pi ih_1 t}\right|^2  \\
        &  \leq  \frac{C}{K} + \frac{C}{K}\sum_{k=1}^K  \sup_t\left|\frac{1}{2(H-1)}\sum_{h_1=0}^{2(2(H-1)-1)}
        g_1(T^{ah_1}x)g_1(T^{a(h_1+k)}x)g_2(T^{b(h_1)}y)g_2(T^{b(h_1+k)}y) e^{2\pi ih_1t}\right| \\
        & \text{ for each K such that } (K+1)^2 < 2(H-1).
        \end{align*}

        Now we can conclude with the path leading to Uniform Wiener Wintner theorem obtained in \cite{AssaniPresser} ( see Lemma 6 and Lemma 7 in this paper). We can use the same method to show that  for each function $V_1,V_2$ bounded by one 
         $$\int \int\limsup_H\sup_t\left| \frac{1}{H} \sum_{h_1=0}^{H-1} V_1(T^{an}x) V_2(T^{bn}y)e^{2\pi int}\right|^2d\mu\otimes \mu \lesssim_{a, b} \min_{i = 1, 2}\mathbf{\vertiii{V_i}_3^2}.  $$
         Therefore we have
         \begin{equation}\label{eq5}
         \int \int\limsup_H \sup_t\left| \frac{1}{H} \sum_{h_1=0}^{H-1} V_1(T^{an}x V_2(T^{bn}y)e^{2\pi int}\right|d\mu\otimes \mu \lesssim_{a, b} \min_{i = 1, 2}\mathbf{\vertiii{V_i}_3}.
         \end{equation}

 As a consequence of (\ref{eq5})  we derive the following estimate.
\begin{align*}
        & \int\int \limsup_H\left(\frac{1}{H^3} \sum_{h_1, h_2, h_3 = 0}^{H-1}\ G^1_{h_1,h_2, h_3}(x) G^2_{h_1,h_2, h_3}(y)\right)^2 d\mu\otimes \mu \\
        &\leq  \frac{C}{K} + \frac{C}{K} \int\int \sum_{k=1}^K \limsup_H \\
        & \sup_t\left|\frac{1}{(2(H-1)}\sum_{h_1= 0}^{2(2(H-1)-1)}
        g_1(T^{ah_1}x)g_1(T^{a(h_1+k)}x)g_2(T^{b(h_1)}x)g_2(T^{b(h_1+k)}x) e^{2\pi ih_1t}\right| d\mu\otimes \mu \\
        & \lesssim_{a, b}\frac{1}{K}\sum_{k=1}^K \min{\left(\vertiii{g_1.g_1(T^{ak})}_3, \vertiii{g_2.g_2(T^{bk})}_3\right)}\\
        & \lesssim_{a, b} \left( \frac{1}{K}\sum_{k=1}^K \min{\left(\vertiii{g_1.g_1(T^{ak})}^8_3, \vertiii{g_2.g_2(T^{bk})}^8_3\right)}\right)^{1/8}
        \end{align*}
     By taking the limit with respect to $K$ we obtain the upper bound $\min{ (\vertiii{g_1}_4, \vertiii{g_2}_4)}.$
     Thus if $g_1$ or $g_2$ belongs to $\mathcal{Z}_3^{\perp}$
       we have shown that the sequence
       $$ \left(\frac{1}{H^3} \sum_{h_1, h_2, h_3 = 0}^{H-1}\ G^1_{h_1,h_2, h_3}(x) G^2_{h_1,h_2, h_3}(y)\right)^2$$ converge a.e. to zero. The dominated convergence theorem allows us to end the proof of this lemma for $k=3.$
        The general case $k\geq 4$ follows similar steps.
      \end{proof}
        So we have shown that we can find a set of full measure $X_1$ such that if one of the functions $f$ or $g$ in the statement of our main theorem belongs to $\mathcal{Z}_{k-1}^{\perp}$ then the averages
        $\frac{1}{N} \sum_{n=0}^{N-1} f(T^{an}x)g(T^{bn}x)c_n)$ converge to zero for each bounded sequence $c_n$ which is a $k$-step nilsequence.
        It remains the case where both functions belong to $\mathcal{Z}_{k-1}.$
        \subsection{ The functions $f$ and $g$ belong to $\mathcal{Z}_{k-1}$}
         We can assume that each function $f$ and $g$ is continuous. It turns out that for each $x$ the sequences $f(T^{an}x)$ and $g(T^{bn}x) $ are themselves $k$-step nilsequences. As the product of two $k-1$ step nilsequences is also a $k-1$ step nilsequence we can conclude that the product
         $f(T^{an}x)g(T^{bn}x)c_n$ is also a $k$-step nilsequence and therefore the convergence of the averages
         $\frac{1}{N} \sum_{n=0}^{N-1} f(T^{an}x) g(T^{bn}x) c_n$ is immediate in this case.
         In the general case the functions are simply in  $L^{\infty}.$  A simple approximation argument allows to derive the same conclusion from the case where the functions are continuous.

         Combining the results in each subsection we have obtained a proof of our main theorem.

         \noindent{\bf Acknowledgments}:
         We thank B. Weiss for bringing this problem to our attention. Thanks also to H. Abdalaoui for his interest in this question.
         
         \noindent{\bf Remark}:
          A proof of Theorem $\ref{mainResult}$ was also announced by P. Zorin-Kranich \cite{PZorinK} independently.



\bibliographystyle{plain}
\bibliography{RM_Bib}

\end{document}